\documentclass{elsarticle}

\usepackage[utf8]{inputenc}
\usepackage{amsmath,amssymb,amsthm,amsfonts}
\usepackage{bm}
\usepackage[hidelinks,colorlinks = true]{hyperref}
\usepackage{enumitem}
\usepackage{blkarray}
\setenumerate{labelindent=0.25\parindent,leftmargin=*}
\usepackage{enumitem}
\usepackage{blkarray}
\usepackage{tikz-cd}
\usepackage{tcolorbox}
\newtcolorbox{mybox}{colback=red!5!white,colframe=red!75!black, sharp corners = all}
\usepackage{bclogo} 
\def\ucsign{\bcpanchant}

\DeclareMathOperator{\rk}{rk}

\DeclareMathOperator{\rad}{rad}

\DeclareMathOperator{\Ker}{ker}
\DeclareMathOperator{\Img}{Im}
\let\det\undefined
\DeclareMathOperator{\det}{det}
\DeclareMathOperator{\Det}{Det}

\DeclareMathOperator{\sgn}{sgn}

\newenvironment{psmallmatrix}
  {\left(\begin{smallmatrix}}
  {\end{smallmatrix}\right)}

\newcommand{\GGG}{=\joinrel =}
\newcommand{\GGGL}{=\joinrel =\joinrel =}

\theoremstyle{plain}
\newtheorem{theorem}{Theorem}[section]
\newtheorem*{theorem*}{Theorem}

\newtheorem{lemma}[theorem]{Lemma}

\newtheorem{corollary}[theorem]{Corollary}

\theoremstyle{definition}
\newtheorem{definition}[theorem]{Definition}
\newtheorem*{definition*}{Definition}

\newtheorem{example}[theorem]{Example}
\newtheorem*{example*}{Example}

\newtheorem{remark}[theorem]{Remark}
\newtheorem{note}[theorem]{Note}
\numberwithin{equation}{section}

\begin{document}

\title{Linear maps preserving the Cullis'
determinant. II}
\author{Alexander Guterman}
\ead{alexander.guterman@biu.ac.il}
\author{Andrey Yurkov\corref{cor1}}
\ead{andrey.yurkov@biu.ac.il}
\cortext[cor1]{Corresponding author}

\affiliation{organization={Department of Mathematics, Bar Ilan University},addressline={Ramat-Gan},postcode={5290002}, country={Israel}}


\begin{abstract}
This paper is the second in the series of papers devoted to the explicit
description of linear maps preserving the Cullis' determinant of rectangular
matrices with entries belonging to an arbitrary ground field which is large
enough.

In this part we solve the linear preserver problem for the Cullis' determinant
defined on the spaces of matrices of size $n\times k$ with
$k \ge 4,\; n \ge k + 2$ and $n + k$ is odd.

In comparison with the case when $n + k$ is even, in this case linear maps
preserving the Cullis' determinant could be singular and are represented
as a sum of two linear maps: first is two-sided matrix multiplication and
second is any linear map whose image consists of matrices, all rows of
which are equal.
\end{abstract}
%
%
%
\maketitle
%
\section{Introduction}
\label{sec1}

The theory of linear maps preserving matrix invariants is a well-known
direction of research in Linear Algebra connected with other fields of
mathematics. It has a long history which dates back to the beginning of
20th century in the works of Frobenius~\cite{GF} and continues at present
time (see~\cite{LAMA199233} for the survey of the results until the end
of 20th century and~\cite{hogben_handbook_2014} for a brief introduction).

The Cullis' determinant is an invariant of rectangular matrix generalizing
the notion of the ordinary determinant of a square matrix which was introduced
by Cullis in~\cite{cullis1913} and studied later by Radi\'{c} in~\cite{radic1966,radic2005,radic1991,radic2008},
Makarewicz, Mozgawa, Pikuta and Sza\l{}kowski in~\cite{makarewicz2014,makarewicz2016,makarewicz2020},
Amiri, Fathy and Bayat in~\cite{amiri2010}, and Nakagami and Yanai in~\cite{NAKAGAMI2007422}.
Thus, the question regarding the description of linear maps preserving
the Cullis' determinant arises naturally. The detailed introduction to
the subject could be found in~\cite{Guterman2025}.

In particular, \cite{Guterman2025} contains the description of linear maps
preserving the Cullis determinant $\Det _{n\,k}$ for the case when
$ k \ge 4$, $n \ge k + 2$ and $n + k$ is even as it is stated in the following
theorem (the necessary definitions are given in Section~\ref{sec:prelim}).

\begin{theorem}[{\cite[Theorem~5.14]{Guterman2025}}]%
\label{MainTheoremEvenIntro}%
Assume that $|{\mathbb F}| > k \ge 4$, $n \ge k + 2$ and $n + k$ is even.
Let
$T\colon {\mathcal M}_{n\,k} ({\mathbb F}) \to {\mathcal M}_{n\,k} ({
\mathbb F})$ be a linear map. Then
$\det _{n\,k} (T(X)) = \det _{n\,k}(X)$ for all
$X \in {\mathcal M}_{n\,k} ({\mathbb F})$ if and only if there exist
$A \in {\mathcal M}_{n\,n}({\mathbb F})$ and
$B \in {\mathcal M}_{k\,k}({\mathbb F})$ 
such that $T(X) = AXB$ for all $X \in {\mathcal M}_{n\, k} ({\mathbb F})$ and
%
\begin{equation*}
\Det _{n\, k} \Bigl(A(|i_{1},\ldots , i_{k}]\Bigr) \cdot \Det _{k}
\Bigl(B\Bigr) = (-1)^{\sum_{\alpha=1}^{k} (i_{\alpha} -\alpha)}\;\;\mbox{for all}\;\; 1 \le i_{1} < \ldots < i_{k} \le n.
\end{equation*}
\end{theorem}

The parity of $n + k$ is significant because if $n + k$ is odd, then there
exist linear maps preserving the Cullis' determinant which do not have
such form (Corollary~\ref{cor:AdditionLinearMapNPLusKOdd}).

However, in this paper we show that the case $(n,k)$ with $n > k$ and
$n + k$ odd is reduced to the case $(n-1,k)$ (Lemma~\ref{lem:LiftDownDetPres}).
Using this reduction we prove the main theorem of this paper providing
the description of linear maps preserving $\Det _{n\,k}$ if $n + k$ is
odd.

\begin{theorem}[Theorem~\ref{thm:MainTheoremCullisNKOdd}]%
\label{thm1.2}
Assume that $|{\mathbb F}| > k \ge 4$, $n \ge k + 2$ and $n + k$ is odd.
Let
$T\colon {\mathcal M}_{n\,k} ({\mathbb F}) \to {\mathcal M}_{n\,k} ({
\mathbb F})$ be a linear map. Then
$\Det _{n\, k} (T(X)) = \Det _{n\,k}(X)$ for all
$X \in {\mathcal M}_{n\,k} ({\mathbb F})$ if and only if there exist
$A \in {\mathcal M}_{n\, n}({\mathbb F})$ and
$B \in {\mathcal M}_{k\, k}({\mathbb F})$ such that
\begin{equation*}
\Det _{n\,k} \Bigl(A(|i_{1},\ldots , i_{k}]\Bigr) \Det _{k} \Bigl(B
\Bigr) = (-1)^{\sum_{\alpha=1}^{k} (i_{\alpha} -\alpha)}\;\;\mbox{for all}\;\; 1 \le i_{1} < \ldots < i_{k} \le n
\end{equation*}
and a linear map
$\phi \colon {\mathcal M}_{n\,k}({\mathbb F}) \to W_{n\,k}$ such that
\begin{equation*}
T(X) = AXB + \phi (X)\;\;\text{for all}\;\; X \in {\mathcal M}_{n\,k} ({
\mathbb F}).
\end{equation*}
Here $W_{n\,k} \subseteq {\mathcal M}_{n\,k}({\mathbb F})$ denotes the
space of matrices, all rows of which are equal, being a radical of
$\Det _{n\,k}$.
\end{theorem}

Our proof of the main theorem relies on the notion of the radical of a
function which was introduced and studied by Waterhouse~\cite{Waterhouse1983}
(Definition~\ref{def:radical} below). The reduction of the case
$(n,k)$ with $k \ge 4, n \ge k + 2$ and $n + k$ odd to the case
$(n-1,k)$ is carried out by finding the radical of $\Det _{n\,k}$ explicitly
and using its properties.

Thus, the results of the current paper and its predecessor~\cite{Guterman2025}
provide the solution to linear preserver problem for the Cullis' determinant
if $k \ge 4, n \ge k + 2$ and the ground field is large enough. The cases
$k = 1$ and $k = 2$ are considered in~\cite{Guterman2025} as well. The
case $n = k + 1$ is considered in~\cite{Guterman2024}, where the first
attempt to find the description of linear maps preserving the Cullis' determinant
is made. The remaining case $k = 3$ requires a special approach because
the proof for the case $k \ge 4$ relies on the following lemma and will
be discussed separately. We remark that in this paper we apply the results
obtained in~\cite{Guterman2025} and we do not use any results from~\cite{Guterman2024}.

\begin{lemma}[{\cite[Lemma~5.10]{Guterman2025}}]%
\label{lem:CullisDeg1Rank1}%
Assume that $|{\mathbb F}| > k \ge 4$, $n \ge k + 2$ and $n + k$ is even.
Let $B \in {\mathcal M}_{n\, k}({\mathbb F})$. If
$\deg _{\lambda} (\Det _{n\, k} (A + \lambda B)) \le 1$
 for all $A \in {\mathcal M}_{n\, k}({\mathbb F})$,
then $\rk (B) \le 1$.
\end{lemma}

The example below shows that this lemma does not hold for $k = 3$.

\begin{example}[{\cite[Example~5.11]{Guterman2025}}]%
\label{cex:K3NKEvenInt}%
Let $n \ge 5$ and
$B =
\begin{psmallmatrix}
1 & 0 & 0 & 0 & -1 & 0 & \cdots & 0
\\
0 & 1 & 0 & -1 & 0 & 0 & \cdots & 0
\\
0 & 0 & 0 & 0 & 0 & 0 & \cdots & 0
\end{psmallmatrix}^{t} \in {\mathcal M}_{n\,3}({\mathbb F})$. Then $\rk (B) = 2$ and
$\deg _{\lambda} (\Det _{n\,k}(A + \lambda B)) \le 1$ for all
$A \in {\mathcal M}_{n\,3}({\mathbb F})$.
\end{example}

Nevertheless, if $k = 3$, then linear maps preserving the Cullis determinant
admit the same description as in the case when $k \ge 4$ and will be studied
separately. %

The paper is organized as follows: in Section~\ref{sec:prelim} we provide
the basic facts regarding $\Det _{n\,k}$; in Section~\ref{sec:kernels}
we find the radical of $\Det _{n\,k}$ assuming that $n + k$ is odd; in Section~\ref{sec:maintheorem}
we reduce the linear preserver problem for $\Det _{n\,k}$ when
$n + k$ is odd, to the linear preserver problem for
$\Det _{(n-1)\,k}$ and prove the main characterization theorem for the case when $n + k$ odd.

\section{Preliminaries}
\label{sec:prelim}

Let us provide the basic notation and definitions used throughout this
paper.

By ${\mathbb F}$ we denote a field without any preliminary restrictions
on its cardinality and characteristic. We will use an inequality
$|{\mathbb F}| > k$ for $k \in \mathbb N$ which means that either
${\mathbb F}$ is infinite or ${\mathbb F}$ is finite and
$|{\mathbb F}| > k$.

We denote by ${\mathcal M}_{n\, k}({\mathbb F})$ the set of all
$n\times k$ matrices with the entries from a certain field
${\mathbb F}$. If $X \in {\mathcal M}_{n\,k}({\mathbb F})$ then by
$x_{i\, j}$ we denote an entry of $X$ lying on the intersection of the
$i$-th row and the $j$-th column of $X$.

$O_{n\, k}\in {\mathcal M}_{n\, k}({\mathbb F})$ denotes the matrix with
all the entries equal to zero.
$I_{n\, n} = I_{n} \in {\mathcal M}_{n\, n}({\mathbb F})$ denotes an identity
matrix. Let us denote by
$E_{ij} \in {\mathcal M}_{n\, k}(\mathbb F)$ a matrix, whose entries are
all equal to zero besides the entry on the intersection of the $i$-th row
and the $j$-th column, which is equal to one. By $x_{i\, j}$ we denote
the element of a matrix $X$ lying on the intersection of its $i$-th row
and $j$-th column. For integers $i, j$ we denote by $\delta _{i\,j}$ a
\emph{Kronecker delta} of $i$ and $j$, which is equal to $1$ if
$i = j$ and equal to $0$ otherwise.

For $A \in {\mathcal M}_{n\,k_{1}}({\mathbb F})$ and
$B \in {\mathcal M}_{n\,k_{2}}({\mathbb F})$ by
$A|B \in {\mathcal M}_{n\,k_{1}+k_{2}}({\mathbb F})$ we denote a block
matrix defined by $A|B =
\begin{pmatrix}
A & B%
\end{pmatrix}
$.

For $A \in {\mathcal M}_{n\,k}$ by
$A^{t} \in {\mathcal M}_{k\,n}({\mathbb F})$ we denote a transpose of the
matrix $A$, i.e. $A^{t}_{i\,j} = A_{j\,i}$ for all $1 \le i \le k$,
$1 \le j \le n$.

We use the notation for submatrices following~\cite{Minc1984} and~\cite{Pierce1979}.
That is, by $A[J_{1}|J_{2}]$ we denote the $|J_{1} |\times |J_{2}|$ submatrix
of $A$ lying on the intersection of rows with the indices from
$J_{1}$ and the columns with the indices from $J_{2}$. By
$A(J_{1}|J_{2})$ we denote a submatrix of $A$ derived from it by striking
out the rows with indices belonging to $J_{1}$ and the columns with the
indices belonging to $J_{2}$. If one of the two index sets is absent, then
it means an empty set, i.e. $A(J_{1}|)$ denotes a matrix derived from
$A$ by striking out from it the rows with indices belonging to
$J_{1}$. We may skip curly brackets, i.e.
$A[1,2|3,4] = A[\{1,2\}|\{3,4\}]$. The notation with mixed brackets is
also used, i.e. $A(|1]$ denotes the first column of the matrix $A$. The
above notations are also used for vectors as well. In this case vectors
are considered as $n\times 1$ or $1 \times n$ matrices.

To avoid the ambiguity, we would like to clarify that the symbol
$\circ $ is used only in the sense of composition of two maps, that is,
if $f \colon A \to B$ and $g \colon B \to C$, then by
$g\circ f\colon A \to C$ we denote a composition of $f$ and $g$ defined
by $(g\circ f)(x) = g(f(x))$ for all $x \in A$.

\begin{definition}
\label{defn2.1}
By $[n]$ we denote the set $\{1, \ldots , n\}$.
\end{definition}
%
\begin{definition}
\label{defn2.2}
By $\mathcal C^{X}_{k}$ we denote the set of injections from $[k]$ to
$X$.
\end{definition}
%
\begin{definition}%
\label{defn2.3}
By $\binom{X}{k}$ we denote the set of the images of injections from
$[k]$ to $X$, i.e. the set of all subsets of $X$ of cardinality $k$.
\end{definition}

\begin{definition}%
\label{defn2.4}
Let $c \in \binom{[n]}{k}$ be equal to $\{i_{1},\ldots , i_{k}\}$, where
$i_{1} < i_{2} < \ldots < i_{k}$, and $1 \le \alpha \le k$ be a natural
number. Then $c(\alpha )$ is defined by
\begin{equation*}
c(\alpha ) = i_{\alpha}.
\end{equation*}
\end{definition}

\begin{definition}%
\label{defn2.5}
Given a set $c \in \binom{[n]}{k}$ we denote by
$\sgn (c) = \sgn _{[n]} (c)$ the number
\begin{equation*}
(-1)^{\sum _{\alpha = 1}^{k} (c(\alpha ) - \alpha )}.
\end{equation*}
\end{definition}

\begin{note}%
\label{note2.6}
$\sgn _{[n]}(c)$ depends only on $c$ and does not depend on $n$.
\end{note}

\begin{definition}%
\label{defn2.7}
Given an injection $\sigma \in \mathcal C^{[n]}_{k}$ we denote by
$\sgn _{n\,k} (\sigma )$ the product
$\sgn (\pi _{\sigma}) \cdot \sgn _{[n]} (c)$, where
$\sgn (\pi _{\sigma})$ is the sign of the permutation
\begin{equation*}
\pi _{\sigma }=
\begin{pmatrix}
i_{1} & \ldots & i_{k}
\\
\sigma (1) & \ldots & \sigma (k)
\end{pmatrix}
,
\end{equation*}
$\{i_{1}, \ldots , i_{k}\} = \sigma ([k])$ and
$i_{1}<i_{2}<\ldots <i_{k}$.
\end{definition}

We omit the subscripts for $\sgn _{[n]}$ and $\sgn _{n\,k}$ if this cannot
lead to a misunderstanding.

\begin{definition}[\cite{NAKAGAMI2007422}, Theorem 13]%
\label{def:CullisDet}
Let $n \ge k$, $X \in {\mathcal M}_{n\,k} (\mathbb F)$. Then the Cullis'
determinant $\Det _{n\, k}(X)$ of $X$ is defined to be the function:
\begin{equation*}
\Det _{n\, k} (X) = \sum _{\sigma \in \mathcal C^{[n]}_{k}} \sgn _{n
\,k} (\sigma ) x_{\sigma (1)\, 1}x_{\sigma (2)\, 2}\ldots x_{\sigma (k)
\, k}.
\end{equation*}
We also denote $\Det _{n\, k} (X)$ as follows
\begin{equation*}
\Det _{n\,k}(X)=
\begin{vmatrix}
x_{1\,1} & \cdots & x_{1\,k}
\\
\vdots & \cdots & \vdots
\\
x_{n\,1} & \cdots & x_{n\,k}
\end{vmatrix}
_{n\,k}.
\end{equation*}
If $n = k$, then we also write $\Det _{k}$ or $\Det $ instead of
$\Det _{n\,k}$ because in this case $\Det _{n\,k}$ is clearly equal to
an ordinary determinant of a square matrix.
\end{definition}

Let us first list the properties of $\Det _{n\, k}$ which are similar to
corresponding properties of the ordinary determinant (see
\cite[\S5, \S27, \S32]{cullis1913} or~\cite{NAKAGAMI2007422}
for detailed proofs).

\begin{theorem}[{\cite[Theorem 13, Theorem 16]{NAKAGAMI2007422}}]
\label{thm2.9}
\begin{enumerate}
\item For $X \in {\mathcal M}_{n}(\mathbb F)$,
$\Det _{n\,n}(X) = \Det (X)$.
\item For $X \in {\mathcal M}_{n\,k}(\mathbb F)$, $\Det _{n\,k}(X)$ is
a multilinear function of columns of $X$.
\item If a matrix $X \in {\mathcal M}_{n\,k}(\mathbb F)$ has two identical
columns or one of its columns is a linear combination of other columns,
then $\Det _{n\,k}(X)$ is equal to zero.
\item For $X \in {\mathcal M}_{n\,k}(\mathbb F)$, interchanging any two
columns of $X$ changes the sign of $\Det _{n\,k}(X)$.
\item Adding a linear combination of columns of $X$ to another column of
$X$ does not change $\Det _{n\,k}(X)$.
\item For $X \in {\mathcal M}_{n\,k}(\mathbb F)$, $\Det _{n\,k}(X)$ can
be calculated using the Laplace expansion along a column of $X$ (see Lemma~\ref{lem:DetNKLaplaceExp}
for precise formulation).
\end{enumerate}
\end{theorem}

\begin{corollary}%
\label{cor:CullisBinomialExpansion}%
Let $n \ge k$, $A, B \in {\mathcal M}_{n\,k} ({\mathbb F})$. Then
%
\begin{multline}
\label{eq:CullisBinomialExpansion}
\Det _{n\,k} (A + \lambda B)
\\
= \sum _{d = 0}^{k}\lambda ^{d} \left ( \sum _{1 \le i_{1} < \ldots < i_{d}
\le k} \Det _{n\,k}\Bigl(A(|1]\Big|\ldots \Big| B(|i_{1}] \Big|
\ldots \Big| B(|i_{d}] \Big| \ldots \Big| A(|k] \Bigr)\right ),
\end{multline}
where the both sides of the equality are considered as formal polynomials
in $\lambda $, i.e. as elements of ${\mathbb F}[\lambda ]$.
\end{corollary}
\begin{proof}%
This is a direct consequence from the multilinearity of
$\Det _{n\,k}$ with respect to the columns of a matrix.
\end{proof}
%
\begin{corollary}%
\label{cor:DegDetABLEQK}%
If $A, B \in {\mathcal M}_{n\,k}({\mathbb F})$, then
$\deg _{\lambda}\left (\Det _{n\,k} (A + \lambda B)\right ) \le k$.
\end{corollary}

\begin{lemma}[{\cite[Corollary~2.8]{Guterman2025}}]%
\label{lem:rightmatrixmult}%
If $X \in {\mathcal M}_{n\,k}({\mathbb F})$ and
$Y \in {\mathcal M}_{k\,k}({\mathbb F})$, then
\begin{equation*}
\Det _{n\,k}(XY) = \Det _{n\,k}(X)\Det _{k}(Y).
\end{equation*}
\end{lemma}

\begin{lemma}[{Cf.~\cite[Theorem~16]{NAKAGAMI2007422}}]%
\label{lem:DetNKLaplaceExp}
Assume that $n \ge k > 1$. For any $n \times k$ matrix
$X = (x_{i\,j})$ the expansion of $\Det _{n\,k}(X)$ along the $j$-th column
is given by
\begin{equation*}
\Det _{n\,k}(X) = \sum _{i = 1}^{n} (-1)^{i+j} x_{i\,j} \Det _{(n-1)
\,(k-1)}\Bigl(X(i|j)\Bigr).
\end{equation*}
\end{lemma}

\begin{lemma}[{Invariance of $\Det _{n\,k}$ under cyclic shifts, cf.~\cite[Theorem 3.5]{amiri2010}}]%
\label{lem:PermKNOddCyclic}%
Assume that $n > k \ge 1$, and $n + k$ is odd. Then for all
$X = (x_{i\,j}) \in {\mathcal M}_{n\,k}({\mathbb F})$ and
$i \in \{1, \ldots , n\}$
\begin{equation*}
(-1)^{(i+1)k}
\begin{vmatrix}
x_{i\,1} & \ldots & x_{i\,k}
\\
\vdots & \vdots & \vdots
\\
x_{n\,1} & \ldots & x_{n\,k}
\\
x_{1\,1} & \ldots & x_{1\,k}
\\
\vdots & \vdots & \vdots
\\
x_{(i-1)\,1} & \ldots & x_{(i-1)\,k}
\end{vmatrix}
_{n\,k} =
\begin{vmatrix}
x_{1\,1} & \ldots & x_{1\,k}
\\
\vdots & \vdots & \vdots
\\
x_{n\,1} & \ldots & x_{n\,k}
\end{vmatrix}
_{n\,k}.
\end{equation*}
Here the matrix on the left-hand side of the equality is obtained from
$X$ by performing the row cyclical shift sending $i$-th row of $X$ to the
first row of the result.
\end{lemma}

\begin{lemma}[{Cf.~\cite[Lemma 3.2]{amiri2010}}]%
\label{lem:AMIRILEMMA32}
Assume that $n \ge k \ge 1$. Let $x_{i\,j} \in {\mathbb F}$ for all
$1 \le i \le n, 1 \le j \le k$. Then
\begin{equation*}
\begin{vmatrix}
x_{1\,1} & \ldots & x_{1\,k}
\\
\vdots & \vdots & \vdots
\\
x_{n\,1} & \ldots & x_{n\,k}
\\
0 & \ldots & 0
\end{vmatrix}
_{(n+1)\,k} =
\begin{vmatrix}
x_{1\,1} & \ldots & x_{1\,k}
\\
\vdots & \vdots & \vdots
\\
x_{n\,1} & \ldots & x_{n\,k}
\end{vmatrix}
_{n\,k}.
\end{equation*}
\end{lemma}

\begin{lemma}[Cf.~{\cite[Lemma~20]{NAKAGAMI2007422}}]%
\label{lem:NAKAGAMILEMMA20} %
Assume that $n > k \ge 1$. Let
$X \in {\mathcal M}_{n\,k}({\mathbb F})$ and
$Y \in {\mathcal M}_{n\,(k+1)}$ is defined by
$Y = X |
\begin{psmallmatrix}
1
\\
\vdots
\\
1%
\end{psmallmatrix}$. Then
\begin{equation*}
\Det _{n\,(k+1)}(Y) =
\begin{cases}
\Det _{n\,k}(X), & \text{$n + k$ is odd},
\\
0, & \text{$n + k$ is even}.
\end{cases}
\end{equation*}
\end{lemma}

\begin{lemma}[{Cf.~\cite[Theorem~3.3]{amiri2010}}]%
\label{lem:AdditionNPLusKOdd}
Assume that $n > k \ge 1$, and $n + k$ is an odd integer. Let
$X \in {\mathcal M}_{n\, k}({\mathbb F})$ be defined by
$X =
\begin{psmallmatrix}
x_{1} & \ldots & x_{k}
\\
\vdots & \ddots & \vdots
\\
x_{1} & \ldots & x_{k}
\end{psmallmatrix}
$ for some $x_{1},\ldots x_{k} \in {\mathbb F}$. Then
\begin{equation*}
\Det _{n\,k} (A + X) = \Det _{n\,k} (A)\;\;\text{for all}\;\;A \in {
\mathcal M}_{n\, k}({\mathbb F})
\end{equation*}
\end{lemma}

\begin{definition}%
\label{def:WNK}%
By $W_{n\,k}\subseteq {\mathcal M}_{n\,k}({\mathbb F})$ we denote
$k$-dimensional vector space consisting of matrices, all rows of which
are equal. That is,
\begin{equation*}
W_{n\,k} = \{
\begin{pmatrix}
y_{1} & \cdots & y_{k}
\\
\vdots & \ddots & \vdots
\\
y_{1} & \cdots & y_{k}%
\end{pmatrix}
\mid y_{1}, \ldots , y_{k} \in {\mathbb F}\}.
\end{equation*}
\end{definition}

\begin{corollary}%
\label{cor:AdditionLinearMapNPLusKOdd}%
Assume that $n > k \ge 1$, and $n + k$ is an odd integer. Let
$\phi \colon {\mathcal M}_{n\,k}({\mathbb F}) \to W_{n\,k}$ and
$T \colon {\mathcal M}_{n\,k}({\mathbb F}) \to {\mathcal M}_{n\,k}({
\mathbb F})$ be linear maps such that
$\Det _{n\, k} (T(X)) = \Det _{n\,k}(X)$ for all
$X \in {\mathcal M}_{n\,k}({\mathbb F})$. Denote by $S$ a linear map
${\mathcal M}_{n\,k}({\mathbb F}) \to {\mathcal M}_{n\,k}({\mathbb F})$
defined by
%
\begin{equation}
\label{eq:AddingKernel}
S(X) = T(X) + \phi (X)\;\;\text{for all}\;\;X \in {\mathcal M}_{n\,k}({
\mathbb F}).
\end{equation}
Then $\Det _{n\, k} (S(X)) = \Det _{n\,k}(X)$ for all $X \in {
\mathcal M}_{n\,k} ({\mathbb F}).$
\end{corollary}

\begin{lemma}[{\cite[Lemma~3.1]{Guterman2025}}]%
\label{lem:TwoSidedMulPreservesDet}%
Let $n \ge k$, $A \in {\mathcal M}_{n\,n}({\mathbb F})$,
$B \in {\mathcal M}_{k\,k}({\mathbb F})$, and let
$T\colon {\mathcal M}_{n\,k}({\mathbb F}) \to {\mathcal M}_{n\,k}({
\mathbb F})$ be a linear map defined by
$T(X) = AXB$ for all $X \in {\mathcal M}_{n\,k}({\mathbb F}).$
Then
$\Det _{n\,k} (T(X)) = \Det _{n\,k}(X)$ for all $X \in {
\mathcal M}_{n\,k} ({\mathbb F})$
if and only if
%
\begin{equation}
\label{eq:TwoSideMul2}
\Det _{n\,k} \Bigl(A(|d]\Bigr)\cdot \Det _{k} \Bigl(B\Bigr) = \sgn (d)
\;\;\text{for all}\;\; d \in \binom{[n]}{k}.
\end{equation}
\end{lemma}

\begin{lemma}[{\cite[Theorem~5.14]{Guterman2025}}]%
\label{lem:MainTheoremEven}%
Assume that $|{\mathbb F}| > k \ge 4, n \ge k + 2$ and $n + k$ is even.
Let
$T\colon {\mathcal M}_{n\, k} ({\mathbb F}) \to {\mathcal M}_{n\, k} ({
\mathbb F})$ be a linear map. Then
$\Det _{n\, k} (T(X)) = \Det _{n\,k}(X)$ for all
$X \in {\mathcal M}_{n\, k} ({\mathbb F})$ if and only if there exist
$A \in {\mathcal M}_{n\, n}({\mathbb F})$ and
$B \in {\mathcal M}_{k\, k}({\mathbb F})$ 
such that $T(X) = AXB$ for all $X \in {\mathcal M}_{n\, k} ({\mathbb F})$ and
%
\begin{equation}
\label{thm:MainTheoremEven:eq1}
\Det _{n\, k} \Bigl(A(|i_{1},\ldots , i_{k}]\Bigr) \cdot \Det _{k}
\Bigl(B\Bigr) = (-1)^{\sum_{\alpha=1}^{k} (i_{\alpha} -\alpha)}\;\;\mbox{for all}\;\; 1 \le i_{1} < \ldots < i_{k} \le n.
\end{equation}
\end{lemma}

The following general fact regarding polynomials over a field will be used
in the sequel.

\begin{lemma}[{\cite[Corollary~2.13]{Guterman2025}}]%
\label{DegFGLessEqual}%
Let $f, g$ be polynomials in $n$ variables over ${\mathbb F}$ such that
either ${\mathbb F}$ is infinite or the degree of $f$ and $g$ in each variable
is less than $|{\mathbb F}|$. If $f$ and $g$ define the same function on
${\mathbb F}^{n}$, then $f = g$.
\end{lemma}

Our proof also relies on the following notion of the radical of a function defined on a finite-dimensional vector space,
introduced in~\cite{Waterhouse1983}.

\begin{definition}[Cf.~{\cite[text at the beginning of Section~1]{Waterhouse1983}}]%
\label{def:radical}%
Let ${\mathbb F}$ be a field, $V$ be a finite-dimensional vector space
over ${\mathbb F}$, and let $f$ be a function from $V$ to
${\mathbb F}$. The \emph{radical} of $f$, denoted by $\rad (f)$ is a subset
of $V$ defined by
\begin{equation*}
\rad (f) = \{\mathbf w \mid f(\mathbf v + \lambda \mathbf w) = f(
\mathbf v)\;\;\text{for all}\;\; \mathbf v \in V,\;\; \lambda \in {
\mathbb F}\}.
\end{equation*}
\end{definition}

\begin{definition}[Cf.~{\cite[text at the beginning of Section~1]{Waterhouse1983}}]%
\label{def:frad}%
If $\rad (f)$ is the radical of $f$ and
$\pi _{\rad}\colon V \to V / \rad (f)$ is the canonical projection, then
$f_{\rad}$ defined as a unique function from $V / \rad (f)$ to
${\mathbb F}$ with trivial radical such that
$f = f_{\rad}\circ \pi _{\rad}$.
\end{definition}

\begin{lemma}[Cf.~{\cite[Proposition~1]{Waterhouse1983}}]%
\label{lem:RadicalCriterion}
$\rad (f)$ is nonzero if and only if there is a noninvertible linear map
$T\colon V\to V$ satisfying $f(T\mathbf v) = f(\mathbf v)$ for all
$\mathbf v$ in $V$.
\end{lemma}

\begin{lemma}%
\label{lem:KernelContainedInRad}%
If $T\colon V \to V$ is a linear map preserving $f$, then
$\Ker T \subseteq \rad (f)$.
\end{lemma}
\begin{proof}%
Indeed, assume that $\mathbf w \in \Ker T$. Then
\begin{equation*}
f(\mathbf v + \lambda \mathbf w)= f(T(\mathbf v + \lambda \mathbf w))
= f(T(\mathbf v) + \lambda T(\mathbf w)) = f(T(\mathbf v) + \mathbf 0)
= f(\mathbf v)
\end{equation*}
for all $\mathbf v \in V$ and $\lambda \in {\mathbb F}$. Therefore,
$\mathbf w \in \rad (f)$.
\end{proof}

\begin{remark}%
\label{rem2.27}
The lemma above is implicitly proved in~\cite{Waterhouse1983} (see~\cite[Proposition~1]{Waterhouse1983}).
\end{remark}

\begin{lemma}[Cf.~{\cite[Proposition~2]{Waterhouse1983}}]%
\label{lem:RadicalInvariant}
Let $V$ be a finite-dimensional vector space over ${\mathbb F}$,
$f\colon V \to {\mathbb F}$ be an arbitrary function. If $f$ has trivial
radical, then the set of all linear maps preserving $f$ forms a group.
If $f$ has nontrivial radical $\rad (f)$, let
$\pi _{rad}\colon V \to V/\rad (f)$ be the projection, and write
$f = f_{\rad} \circ \pi _{\rad}$. Then the linear map
$T \colon V \to V$ preserves $f$ if and only if
%
%
\begin{enumerate}[label=(\alph*),ref=(\alph*)]
\item
\label{lem:RadicalInvariant:part1}%
$T(\rad (f)) \subseteq \rad (f)$ and
\item
\label{lem:RadicalInvariant:part2}%
the map $T_{\rad} \colon V / \rad (f) \to V / \rad (f)$ that is given by
$T_{\rad}(\pi _{\rad}(\mathbf v)) = \pi _{\rad}(T(\mathbf v))$ (and is
well defined because of~(a)) preserves
$f_{\rad}$.
\end{enumerate}
\end{lemma}

\section{The radical of $\Det _{n\,k}$ if $n + k$ is odd}
\label{sec:kernels}

\begin{remark}%
\label{rem3.1}
The approach used in this section is similar to one used in~\cite[Section~4]{Guterman2025}
but differs in details.
\end{remark}

\begin{lemma}%
\label{lem:DetNKXiEqX1}%
Assume that $n > k \ge 2$, and $n+k$ is odd. Let
$x_{1}, \ldots , x_{n} \in {\mathbb F}$ and
\begin{equation*}
X =
\begin{pmatrix}
x_{1} & 0 & 0 & \cdots & 0
\\
x_{2} & 0 & 0 & \cdots & 0
\\
x_{3} & 1 & 0 &\cdots & 1
\\
x_{4} & 0 & 1 &\cdots & 1
\\
\vdots & \vdots & \vdots & \ddots & \vdots
\\
x_{k} & 0 & 0 & \cdots & 1
\\
\vdots & \vdots & \vdots & \ddots & \vdots
\\
x_{n} & 0 & 0 & \cdots & 1
\end{pmatrix}
\in {\mathcal M}_{n\,k}({\mathbb F}).
\end{equation*}
Then
\begin{equation*}
\Det _{n\,k}(X) = (-1)^{k-1}(x_{1} - x_{2}).
\end{equation*}
\end{lemma}
\begin{proof}
Observe that
%
\begin{equation}
\label{lem:DetNKXiEqX1:eq1}
\begin{vmatrix}
x_{1} & 0 & 0 & \cdots & 0
\\
x_{2} & 0 & 0 & \cdots & 0
\\
x_{3} & 1 & 0 &\cdots & 1
\\
x_{4} & 0 & 1 &\cdots & 1
\\
\vdots & \vdots & \vdots & \ddots & \vdots
\\
x_{k} & 0 & 0 & \cdots & 1
\\
\vdots & \vdots & \vdots & \ddots & \vdots
\\
x_{n} & 0 & 0 & \cdots & 1
\end{vmatrix}
_{n\,k} =
\begin{vmatrix}
x_{1} & 0 & 0 & \cdots & 1
\\
x_{2} & 0 & 0 & \cdots & 1
\\
x_{3} & 1 & 0 &\cdots & 1
\\
x_{4} & 0 & 1 &\cdots & 1
\\
\vdots & \vdots & \vdots & \ddots & \vdots
\\
x_{k} & 0 & 0 & \cdots & 1
\\
\vdots & \vdots & \vdots & \ddots & \vdots
\\
x_{n} & 0 & 0 & \cdots & 1
\end{vmatrix}
_{n\,k} -
\begin{vmatrix}
x_{1} & 0 & 0 & \cdots & 1
\\
x_{2} & 0 & 0 & \cdots & 1
\\
x_{3} & 1 & 0 &\cdots & 0
\\
x_{4} & 0 & 1 &\cdots & 0
\\
\vdots & \vdots & \vdots & \ddots & \vdots
\\
x_{k} & 0 & 0 & \cdots & 0
\\
\vdots & \vdots & \vdots & \ddots & \vdots
\\
x_{n} & 0 & 0 & \cdots & 0
\end{vmatrix}
_{n\,k}.
\end{equation}

Consider the first term in the difference~\eqref{lem:DetNKXiEqX1:eq1}.
Lemma~\ref{lem:NAKAGAMILEMMA20} implies that it is zero since
$n + (k - 1)$ is even.

Now consider the second term in the difference~\eqref{lem:DetNKXiEqX1:eq1}.
Using the Laplace expansion along the second column for $k-2$ times we
obtain that
%
\begin{equation}
\label{lem:DetNKXiEqX1:eq3}
\begin{vmatrix}
x_{1} & 0 & 0 & \cdots & 1
\\
x_{2} & 0 & 0 & \cdots & 1
\\
x_{3} & 1 & 0 &\cdots & 0
\\
x_{4} & 0 & 1 &\cdots & 0
\\
\vdots & \vdots & \vdots & \ddots & \vdots
\\
x_{k} & 0 & 0 & \cdots & 0
\\
\vdots & \vdots & \vdots & \ddots & \vdots
\\
x_{n} & 0 & 0 & \cdots & 0
\end{vmatrix}
_{n\,k} = (-1)^{s}
\begin{vmatrix}
x_{1} & 1
\\
x_{2} & 1
\\
x_{k+1} & 0
\\
\vdots & \vdots
\\
x_{n} & 0
\end{vmatrix}
_{(n-k+2)\,2},
\end{equation}
where
\begin{equation*}
s = {(2+3) + \ldots + (2+3)} = (2+3)(k-2).
\end{equation*}
Hence,
%
\begin{equation}
\label{lem:DetNKXiEqX1:eq4}
(-1)^{s} = (-1)^{k-2}.
\end{equation}

The Laplace expansion along the second column yields that
%
\begin{multline}
\label{lem:DetNKXiEqX1:eq2}
\begin{vmatrix}
x_{1} & 1
\\
x_{2} & 1
\\
x_{k+1} & 0
\\
\vdots & \vdots
\\
x_{n} & 0
\end{vmatrix}_{(n-k+2)\,2} = (-1)^{2 + 1}
\begin{vmatrix}
x_{2}
\\
x_{k+1}
\\
\vdots
\\
x_{n}
\end{vmatrix}_{(n-k+1)\,1} + (-1)^{2 + 2}
\begin{vmatrix}
x_{1}
\\
x_{k+1}
\\
\vdots
\\
x_{n}
\end{vmatrix}_{(n-k+1)\,1}
\\
= -\left (x_{2} + \sum _{i=1}^{n-k} (-1)^{i} x_{k+i}\right ) + \left (x_{1}
+ \sum _{i=1}^{n-k} (-1)^{i} x_{k+i}\right ) = x_{1} - x_{2}.
\end{multline}

Thus, after making the substitution of~\eqref{lem:DetNKXiEqX1:eq2},
\eqref{lem:DetNKXiEqX1:eq3} and~\eqref{lem:DetNKXiEqX1:eq4} into~\eqref{lem:DetNKXiEqX1:eq1},
we obtain that
\begin{equation*}
\Det _{n\,k}(X) = 0 - (-1)^{k-2}(x_{1} - x_{2}) = (-1)^{k-1}(x_{1} - x_{2}).\qedhere
\end{equation*}
\end{proof}

\begin{lemma}%
\label{lem:NKEvenAdditionPreserverHasZero11Odd}%
Assume that $n > k \ge 2$, $n + k$ is odd and $|{\mathbb F}| > k$. Let
$Y \in {\mathcal M}_{n\, k} ({\mathbb F})$ be such that
%
\begin{equation}
\label{lem:NKEvenAdditionPreserverHasZero11Odd:cond}
\Det _{n\, k} (A + \lambda Y) = \Det _{n\, k} (A)\;\;\text{for all}\;
\;A \in {\mathcal M}_{n\, k}({\mathbb F})\;\;\text{and}\;\;\lambda
\in {\mathbb F}.
\end{equation}
Then $y_{1\,1} - y_{2\,1} = 0$.
\end{lemma}
\begin{proof}%
Let $A \in {\mathcal M}_{n\, k}({\mathbb F})$ be defined by
\begin{equation*}
A = E_{3\,2} + E_{4\,3} + \ldots + E_{k\, (k-1)} + E_{3\, k} +
\ldots + E_{n\, k} =
\begin{pmatrix}
0 & 0 & 0 & \cdots & 0
\\
0 & 0 & 0 & \cdots & 0
\\
0 & 1 & 0 &\cdots & 1
\\
0 & 0 & 1 &\cdots & 1
\\
\vdots & \vdots & \vdots & \ddots & \vdots
\\
0 & 0 & 0 & \cdots & 1
\\
\vdots & \vdots & \vdots & \ddots & \vdots
\\
0 & 0 & 0 & \cdots & 1
\end{pmatrix}
.
\end{equation*}

It follows from Corollary~\ref{cor:CullisBinomialExpansion} that there
is a polynomial $P \in {\mathbb F}[\lambda ]$ with
$\deg _{\lambda}(P) \le k$ such that
$P = \Det _{n\, k} (A + \lambda Y)$. Let
$a_{0}, \ldots , a_{k} \in {\mathbb F}$ be the coefficients of $P$. That
is,
\begin{equation*}
P = \Det _{n\, k} (A + \lambda Y) = a_{0} + a_{1}\lambda + \ldots + a_{k}
\lambda ^{k}.
\end{equation*}
On the one hand, the condition~\eqref{lem:NKEvenAdditionPreserverHasZero11Odd:cond}
implies that
\begin{equation*}
a_{0} + a_{1}\lambda + \ldots + a_{k} \lambda ^{k} = \Det _{n\,k} (A)
\;\;\text{for all}\;\; \lambda \in {\mathbb F}.
\end{equation*}
Therefore,
%
\begin{equation}
\label{lem:NKEvenAdditionPreserverHasZero11Odd:eq3}
a_{1} = \ldots = a_{k} = 0
\end{equation}
by Lemma~\ref{DegFGLessEqual} because $|{\mathbb F}| > k$.

On the other hand, the definition of $a_{1}$ implies that
\begin{equation*}
a_{1} = \sum _{1 \le i_{1} \le k} \Det _{n\,k}\Bigl(A(|1]\Big|\ldots
\Big| Y(|i_{1}] \Big| \ldots \Big| A(|k] \Bigr)
\end{equation*}
by~\eqref{eq:CullisBinomialExpansion}. Since $A(|1]$ is a zero column,
this implies that
\begin{multline*}
a_{1} = \sum _{1 \le i_{1} \le k} \Det _{n\,k}\Bigl(A(|1]\Big|\ldots
\Big| Y(|i_{1}] \Big| \ldots \Big| A(|k] \Bigr)
\\
= \Det _{n\,k} \Bigl(Y(|1] \Big| A(|2] \Big| \ldots \Big| A(|k]\Bigr) +
\sum _{2 \le i_{1} \le k} \Det _{n\,k}\Bigl(A(|1]\Big|\ldots \Big| Y(|i_{1}]
\Big| \ldots \Big| A(|k] \Bigr)
\\
= \Det _{n\,k} \Bigl(Y(|1] \Big| A(|2] \Big| \ldots \Big| A(|k]\Bigr) +
\sum _{2 \le i_{1} \le k} \Det _{n\,k}\Bigl(0\Big|\ldots \Big| Y(|i_{1}]
\Big| \ldots \Big| A(|k] \Bigr)
\\
= \Det _{n\,k} \Bigl(Y(|1] \Big| A(|2] \Big| \ldots \Big| A(|k]\Bigr) +
\sum _{2 \le i_{1} \le k} 0
= \Det _{n\,k} \Bigl(Y(|1] \Big| A(|2] \Big| \ldots \Big| A(|k]\Bigr).
\end{multline*}

By substituting of the right-hand side of the above equality in~\eqref{lem:NKEvenAdditionPreserverHasZero11Odd:eq3}
we obtain that
%
\begin{equation}
\label{lem:NKEvenAdditionPreserverHasZero11Odd:eq1}
\Det _{n\,k} \Bigl(Y(|1] \Big| A(|2] \Big| \ldots \Big| A(|k]\Bigr) = 0.
\end{equation}

Now note that
$\Bigl(Y(|1] \Big| A(|2] \Big| \ldots \Big| A(|k]\Bigr)$ has the form described
in the statement of Lemma~\ref{lem:DetNKXiEqX1} for
$x_{1} = y_{1\, 1}, \ldots , x_{n} = y_{n\,1}$. This implies that%
%
\begin{equation}
\label{lem:NKEvenAdditionPreserverHasZero11Odd:eq2}
\Det _{n\, k} \Bigl(Y(|1] \Big| A(|2] \Big| \ldots \Big| A(|k]\Bigr) =
(-1)^{k-1}(y_{1\,1} - y_{2\,1}).
\end{equation}

The equalities~\eqref{lem:NKEvenAdditionPreserverHasZero11Odd:eq1} and~\eqref{lem:NKEvenAdditionPreserverHasZero11Odd:eq2}
allow us to conclude that
\begin{equation*}
(-1)^{k-1}(y_{1\,1} - y_{2\,1}) = \Det _{n\, k} \Bigl(Y(|1] \Big| A(|2]
\Big| \ldots \Big| A(|k]\Bigr) = 0,
\end{equation*}
and consequently
\begin{equation*}
y_{1\,1} - y_{2\,1} = 0.\qedhere
\end{equation*}
\end{proof}

\begin{definition}%
\label{def:SCSDef}%
Let $n \ge k$ and $1 \le i \le n$, $1 \le j \le k$. By
${\mathbf S}_{i\,j}$ we denote a linear map on
${\mathcal M}_{n\,k}({\mathbb F})$ defined by
\[
{\mathbf S}_{i\,j}
\begin{pmatrix}
x_{1\,1} & \cdots & x_{1\,k}
\\
\vdots & \ddots & \vdots
\\
x_{n\,1} & \cdots & x_{n\,k}
\end{pmatrix}
= (-1)^{i+1}
\begin{pmatrix}
(-1)^{1 - \delta _{1\,j}}x_{i\, j} & \ldots & x_{i\,1} & \ldots & x_{i
\, k}
\\
\vdots & \ddots & \vdots & \ddots & \vdots
\\
(-1)^{1 - \delta _{1\,j}}x_{n\, j} & \ldots & x_{n\,1} & \ldots & x_{n
\, k}
\\
(-1)^{1 - \delta _{1\,j}}x_{1\, j} & \ldots & x_{1\,1} & \ldots & x_{1
\, k}
\\
\vdots & \ddots & \vdots & \ddots & \vdots
\\
(-1)^{1 - \delta _{1\,j}}x_{i-1\, j} & \ldots & x_{i-1\,1} & \ldots & x_{i-1
\, k}
\end{pmatrix}.
\]
That is, ${\mathbf S}_{i\,j}(X)$ is obtained from $X$ by performing the
following sequence of operations:
\begin{enumerate}
\item
\label{lst:SIJSeq:it1}%
the row cyclical shift sending $i$-th row of $X$ to the first row of the
result;
\item exchanging the first and the $j$-th column;
\item
\label{lst:SIJSeq:it3}%
multiplying the first column by $(-1)^{1 - \delta _{1\,j}}$;
\item multiplying all the entries by $(-1)^{i+1}$.
\end{enumerate}
\end{definition}

\begin{lemma}%
\label{lem:ReduceTo11ByShiftOdd}%
Assume that $n > k$ and $n + k$ is odd. Then ${\mathbf S}_{i\,j}$ is an
invertible linear map preserving $\Det _{n\,k}$ for all
$1 \le i \le n$, $1 \le j \le k$.
\end{lemma}
\begin{proof}%
Invertibility of ${\mathbf S}_{i\,j}$ follows directly from the definition.
Assume that $X = (x_{i\,j}) \in {\mathcal M}_{n\,k}({\mathbb F})$. Let
$X' \in {\mathcal M}_{n\,k}({\mathbb F})$ be defined by
\begin{equation*}
X' =
\begin{pmatrix}
(-1)^{1 - \delta _{1\,j}}x_{i\, j} & \ldots & x_{i\,1} & \ldots & x_{i
\, k}
\\
\vdots & \ddots & \vdots & \ddots & \vdots
\\
(-1)^{1 - \delta _{1\,j}}x_{n\, j} & \ldots & x_{n\,1} & \ldots & x_{n
\, k}
\\
(-1)^{1 - \delta _{1\,j}}x_{1\, j} & \ldots & x_{1\,1} & \ldots & x_{1
\, k}
\\
\vdots & \ddots & \vdots & \ddots & \vdots
\\
(-1)^{1 - \delta _{1\,j}}x_{i-1\, j} & \ldots & x_{i-1\,1} & \ldots & x_{i-1
\, k}
\end{pmatrix}
.
\end{equation*}
That is, $X'$ is obtained from $X$ by performing the operations
1--3 from the sequence in the definition
of ${\mathbf S}_{i\,j}$ (Definition~\ref{def:SCSDef}). The definition of
$X'$ implies that $X' = (-1)^{i+1}{\mathbf S}_{i\,j}(X)$.

Let $C_{S} \in {\mathcal M}_{k\,k}({\mathbb F})$ be a matrix such that
the linear map $Y \mapsto YC_{S}$ on
${\mathcal M}_{n\,k}({\mathbb F})$ exchanges the first and the $j$-th column
of the matrix and multiplies its first column by
$(-1)^{1 - \delta _{1\,j}}$. That is, $C_{S}$ is the permutation matrix
corresponding to a transposition/trivial permutation $(1j)$ which first
column is multiplied by $(-1)^{1 - \delta _{1\,j}}$. This implies in particular
that $\Det _{k}(C_{S}) = 1$.

On the one hand, we have the following sequence of equalities
%
\begin{equation}
\label{lem:ReduceTo11ByShiftOdd:eq1}
\Det _{n\,k}({\mathbf S}_{i\,j}(X)) = \Det _{n\,k}((-1)^{i+1}X') = (-1)^{(i+1)k}
\Det _{n\,k}(X')
\end{equation}
where the last equality follows from the multilinearity of
$\Det _{n\,k}$ with respect to columns of $X'$.

On the other hand,
\begin{equation*}
XC_{S} =
\begin{pmatrix}
(-1)^{1 - \delta _{1\,j}}x_{1\, j} & \ldots & x_{1\,1} & \ldots & x_{1
\, k}
\\
\vdots & \ddots & \vdots & \ddots & \vdots
\\
(-1)^{1 - \delta _{1\,j}}x_{n\, j} & \ldots & x_{n\,1} & \ldots & x_{n
\, k}
\end{pmatrix}
,
\end{equation*}
by the definition of $C_{S}$. Hence,
%
\begin{multline}
\label{lem:ReduceTo11ByShiftOdd:eq2}
(-1)^{(i+1)k}\Det _{n\,k}(X')
\\
= (-1)^{(i+1)k}
\begin{vmatrix}
(-1)^{1 - \delta _{1\,j}}x_{i\, j} & \ldots & x_{i\,1} & \ldots & x_{i
\, k}
\\
\vdots & \ddots & \vdots & \ddots & \vdots
\\
(-1)^{1 - \delta _{1\,j}}x_{n\, j} & \ldots & x_{n\,1} & \ldots & x_{n
\, k}
\\
(-1)^{1 - \delta _{1\,j}}x_{1\, j} & \ldots & x_{1\,1} & \ldots & x_{1
\, k}
\\
\vdots & \ddots & \vdots & \ddots & \vdots
\\
(-1)^{1 - \delta _{1\,j}}x_{i-1\, j} & \ldots & x_{i-1\,1} & \ldots & x_{i-1
\, k}
\end{vmatrix}
= \Det _{n\,k}(XC_{S}),
\end{multline}
where the last equality follows from Lemma~\ref{lem:PermKNOddCyclic}. In
addition, since $\Det _{k}(C_{S}) = 1$, then
%
\begin{equation}
\label{lem:ReduceTo11ByShiftOdd:eq3}
\Det _{n\,k}(XC_{S}) = \Det _{n\,k}(X)
\end{equation}
by Lemma~\ref{lem:rightmatrixmult}.

Therefore,
\begin{equation*}
\Det _{n\,k}({\mathbf S}_{i\,j}(X)) = (-1)^{(i+1)k}\Det _{n\,k}(X') =
\Det _{n\,k}(XC_{S}) = \Det _{n\,k}(X)
\end{equation*}
by \eqref{lem:ReduceTo11ByShiftOdd:eq1} and \eqref{lem:ReduceTo11ByShiftOdd:eq2},
\eqref{lem:ReduceTo11ByShiftOdd:eq3}. From this we conclude that
${\mathbf S}_{i\,j}$ preserves $\Det _{n\,k}$.
\end{proof}

\begin{lemma}%
\label{lem:NKOddRadDetNK}%
Assume that $|{\mathbb F}| > k$, $n > k \ge 2$ and $n + k$ is odd. Then
%
\begin{equation}
\label{lem:NKOddRadDetNK:eqq}
\rad (\Det _{n\,k}) = W_{n\,k}.
\end{equation}
\end{lemma}
\begin{proof}
The inclusion $W_{n\,k} \subseteq \rad (\Det _{n\,k})$ follows from Lemma~\ref{lem:AdditionNPLusKOdd}.
Let us prove the inclusion
%
\begin{equation}
\label{lem:NKOddRadDetNK:eq1}
\rad (\Det _{n\,k}) \subseteq W_{n\,k}.
\end{equation}
For this, assume that $Y \in \rad (\Det _{n\,k})$ and show that
$y_{i\,j} - y_{i+1\,j}= 0$ for all $1 \le i < n, 1 \le j \le k$. This clearly
implies~\eqref{lem:NKOddRadDetNK:eq1}.

Let $1 \le i < n$ and $1 \le j \le k$. Consider a matrix
${\mathbf S}_{i\,j}(Y)$, where ${\mathbf S}_{i\,j}$ is a linear map defined
in Definition~\ref{def:SCSDef}. This definition implies that
%
\begin{equation}
\label{lem:NKEvenAdditionPreserverIsZeroOdd:eq1}
\left ({\mathbf S}_{i\,j}(Y)\right )_{1\,1} = (-1)^{i - \delta _{1\,j}}y_{i
\,j}\quad \text{and}\quad \left ({\mathbf S}_{i\,j}(Y)\right )_{2\,1} =
(-1)^{i - \delta _{1\,j}}y_{(i+1)\,j}.
\end{equation}
Also
\begin{multline}
\Det _{n\,k} (A + \lambda {\mathbf S}_{i\,j}(Y)) = \Det _{n\,k} ({
\mathbf S}_{i\,j}({\mathbf S}_{i\,j}^{-1}(A)) + \lambda {\mathbf S}_{i
\,j}(Y))
= \Det _{n\,k} ({\mathbf S}_{i\,j}({\mathbf S}_{i\,j}^{-1}(A) +
\lambda Y))\\
= \Det _{n\,k} ({\mathbf S}_{i\,j}^{-1}(A) + \lambda Y)
= \Det _{n\,k} ({\mathbf S}_{i\,j}^{-1}(A)) = \Det _{n\,k}(A)\;\;\mbox{for all}\;\;A \in {\mathcal M}_{n\,k}({\mathbb F}),\,\lambda \in {\mathbb F}
\end{multline}

Therefore, ${\mathbf S}_{i\,j}(Y)$ satisfies the conditions of Lemma~\ref{lem:NKEvenAdditionPreserverHasZero11Odd}.
It implies that
%
\begin{equation}
\label{lem:NKEvenAdditionPreserverIsZeroOdd:eq2}
\left ({\mathbf S}_{i\,j}(Y)\right )_{1\,1} - \left ({\mathbf S}_{i\,j}(Y)
\right )_{2\,1} = 0.
\end{equation}
By substituting~\eqref{lem:NKEvenAdditionPreserverIsZeroOdd:eq1} to \eqref{lem:NKEvenAdditionPreserverIsZeroOdd:eq2} we obtain that
\begin{equation*}
(-1)^{i - \delta _{1\,j}}\left (y_{i\,j} - y_{(i+1)\,j}\right ) =
\left ({\mathbf S}_{i\,j}(Y)\right )_{1\,1} - \left ({\mathbf S}_{i\,j}(Y)
\right )_{2\,1}= 0
\end{equation*}
and consequently
\begin{equation*}
y_{i\,j} - y_{(i+1)\,j} = 0.
\end{equation*}
This equality holds for all $1 \le i < n$ and $1 \le j \le k$. Therefore,
$Y \in W_{n\,k}$, and consequently the inclusion~\eqref{lem:NKOddRadDetNK:eq1}
holds.
\end{proof}

\section{The proof of the main theorem}
\label{sec:maintheorem}

In this section we describe a relationship between linear
$\Det _{n\,k}$-preservers and $\Det _{(n-1)\,k}$-preservers. This relationship
allows us to reduce the linear preserver problem for $\Det _{n\,k}$ to
the linear preserver problem for $\Det _{(n-1)\,k}$ as it is done in Lemma~\ref{lem:FromNKOddToNKEvenIfKerW}.
In conclusion we obtain a proof of the main theorem of this paper.

The relationship is established through the space
${\mathcal M}^{0}_{n\,k}({\mathbb F})$ defined below which is possible
to identify with ${\mathcal M}_{(n-1)\,k}({\mathbb F})$ by the linear maps
converting $\Det _{n\,k}$ to $\Det _{n-1\,k}$ and vice versa. These linear
maps are defined in Definition~\ref{def:Lminus} and Definition~\ref{def:Lplus},
respectively.

\begin{definition}
\label{defn4.1}
By
${\mathcal M}^{0}_{n\,k}({\mathbb F}) \subseteq {\mathcal M}_{n\,k}({
\mathbb F})$ we denote a linear space of all matrices such that their last
row equal to zero, i.e.
\begin{equation*}
{\mathcal M}^{0}_{n\,k}({\mathbb F}) = \{X \in {\mathcal M}_{n\,k}({
\mathbb F}) \mid x_{n\,j} = 0, 1 \le j \le k\}.
\end{equation*}
\end{definition}

\begin{definition}%
\label{def:Lminus}
For every matrix $X \in {\mathcal M}_{n\,k}(\mathbb F)$ by
$L^{-}(X) \in {\mathcal M}_{n-1\,k} (\mathbb F)$ we denote the matrix defined
by
\begin{equation*}
L^{-}(X) =
\begin{pmatrix}
X[1|) - X[n|)
\\
\vdots
\\
X[n-1|) - X[n|)%
\end{pmatrix}
.
\end{equation*}
\end{definition}

\begin{lemma}%
\label{claim:LminusMatrixForm}%
$L^{-}(X) = M^{-}X$ for all $X \in {\mathcal M}_{n\,k}({\mathbb F})$, where
\begin{equation*}
M^{-} =
\begin{pmatrix}
1 & 0 & \ldots & 0 & -1
\\
0 & 1 & \ldots & 0 & -1
\\
\vdots & \vdots & \ddots & \vdots & \vdots
\\
0 & 0 & \ldots & 1 & -1
\end{pmatrix}
\in {\mathcal M}_{n-1\,n}({\mathbb F}).
\end{equation*}
\end{lemma}
\begin{proof}%
The proof is carried out by the direct computation.
\end{proof}

\begin{lemma}%
\label{lem:DetNKLMinus}%
Assume that $n > k \ge 1$ and $n + k$ is odd. Then
$\Det _{n\,k}(X){ = }\Det _{(n-1)\,k}(L^{-}(X))$ for all
$X \in {\mathcal M}_{n\,k}({\mathbb F})$.
\end{lemma}
\begin{proof}%
Lemma~\ref{lem:AdditionNPLusKOdd} implies that
$\Det _{n\,k} (X) = \Det _{n\,k}(X')$, where
\begin{equation*}
X' =
\begin{pmatrix}
X[1|) - X[n|)
\\
\vdots
\\
X[n-1|) - X[n|)
\\
0%
\end{pmatrix}
.
\end{equation*}
Then $\Det _{n\,k} (X') = \Det _{n-1\,k} (L^{-}(X))$ by Lemma~\ref{lem:AMIRILEMMA32}
because the last row of $X'$ is zero and
$X'[1,\ldots , n-1|) = L^{-}(X)$.
\end{proof}

\begin{definition}%
\label{def:Lplus}
For every matrix $X \in {\mathcal M}_{n-1\,k}(\mathbb F)$ by
$L^{+}(X) \in {\mathcal M}^{0}_{n\,k} (\mathbb F)$ we denote the matrix
obtained from $X$ by adjoining to it the zero row as the last row; that
is
\begin{equation*}
L^{+}(X) =
\begin{pmatrix}
X
\\
0
\end{pmatrix}
.
\end{equation*}
\end{definition}

\begin{lemma}%
\label{claim:LplusMatrixForm}%
$L^{+}(X) = M^{+}X$ for all
$X \in {\mathcal M}_{n-1\,k}({\mathbb F})$, where
\begin{equation*}
M^{+} =
\begin{pmatrix}
1 & 0 & \ldots & 0
\\
0 & 1 & \ldots & 0
\\
\vdots & \vdots & \ddots & \vdots
\\
0 & 0 & \ldots & 1
\\
0 & 0 & \ldots & 0
\end{pmatrix}
\in {\mathcal M}_{n\,n-1}({\mathbb F}).
\end{equation*}
\end{lemma}
\begin{proof}%
The proof is carried out by the direct computation.
\end{proof}

\begin{lemma}%
\label{lem:DetNKLPlus}%
Assume that $n > k \ge 1$. Then
$\Det _{(n-1)\,k}(Y) = \Det _{n\,k}(L^{+}(Y))$ for all
$Y \in {\mathcal M}_{(n-1)\,k}({\mathbb F})$.
\end{lemma}
\begin{proof}%
It directly follows from Lemma~\ref{lem:AMIRILEMMA32}.
\end{proof}

The maps $L^{+}$ and $L^{-}$ provide a correspondence between linear maps
preserving $\Det _{n\,k}$ and linear map preserving
$\Det _{(n-1)\,k}$ as it is stated in the next two lemmas.

\begin{lemma}%
\label{lem:LiftDownDetPres}%
Assume that $n > k \ge 1$ and $n + k$ is odd. Let $T$ be a linear map on
${\mathcal M}_{n\, k} ({\mathbb F})$ such that
$\Det _{n\, k} (T(X)) = \Det _{n\,k}(X)$ for all
$X \in {\mathcal M}_{n\,k}({\mathbb F})$. If a linear map $S$ on
${\mathcal M}_{(n-1)\,k}({\mathbb F})$ is defined by
\begin{equation*}
S = L^{-} \circ T \circ L^{+},
\end{equation*}
then $\Det _{(n-1)\, k} (S(Y)) = \Det _{(n-1)\,k}(Y)$ for all
$Y \in {\mathcal M}_{(n-1)\,k}({\mathbb F})$.
\end{lemma}
\begin{proof}%
Let $Y \in {\mathcal M}_{(n-1)\,k}({\mathbb F})$. Then
%
\begin{equation}
\label{lem:LiftDownDetPres:eq1}
\Det _{n\,k}(L^{+}(Y)) = \Det _{(n-1)\,k}(Y)
\end{equation}
by Lemma~\ref{lem:DetNKLPlus}. Also
%
\begin{equation}
\label{lem:LiftDownDetPres:eq2}
\Det _{n\,k}(T(L^{+}(Y))) = \Det _{n\,k}(L^{+}(Y))
\end{equation}
because $T$ preserves $\Det _{n\,k}$.

Lemma~\ref{lem:DetNKLMinus} implies that
%
\begin{equation}
\label{lem:LiftDownDetPres:eq3}
\Det _{n\,k}(T(L^{+}(Y))) = \Det _{(n-1)\,k}(L^{-}(T(L^{+}(Y))).
\end{equation}
The equalities~\eqref{lem:LiftDownDetPres:eq1},
\eqref{lem:LiftDownDetPres:eq2} and~\eqref{lem:LiftDownDetPres:eq3} aligned
together imply that
\begin{equation*}
\Det _{(n-1)\,k}(Y) = \Det _{(n-1)\,k}(L^{-}(T(L^{+}(Y))))\;\;
\text{for all}\;\; Y \in {\mathcal M}_{(n-1)\,k}({\mathbb F}).
\end{equation*}
Since $ L^{-} \circ T \circ L^{+} = S$ by the definition of $S$, we conclude
that $S$ preserves $\Det _{(n-1)\,k}$.
\end{proof}

\begin{lemma}%
\label{lem:LiftUpDetPres}%
Assume that $n > k \ge 1$ and $n + k$ is odd. Let $S$ be a linear map on
${\mathcal M}_{(n-1)\, k} ({\mathbb F})$ such that
$\Det _{(n-1)\, k} (S(Y)) = \Det _{(n-1)\,k}(Y)$ for all %
$Y \in {\mathcal M}_{(n-1)\,k}({\mathbb F})$. If a linear map $T$ on
${\mathcal M}_{n\,k}({\mathbb F})$ is defined by
\begin{equation*}
T = L^{+} \circ S \circ L^{-},
\end{equation*}
then $\Det _{n\, k} (T(X)) = \Det _{n\,k}(X)$ for all
$X \in {\mathcal M}_{n\,k}({\mathbb F})$
\end{lemma}
\begin{proof}%
The proof is similar to the proof of the previous lemma. Let
$X \in {\mathcal M}_{n\,k}({\mathbb F})$. Then
%
\begin{equation}
\label{lem:LiftUpDetPres:eq1}
\Det _{(n-1)\,k}(L^{-}(X)) = \Det _{n\,k}(X)
\end{equation}
by Lemma~\ref{lem:DetNKLMinus}. Also
%
\begin{equation}
\label{lem:LiftUpDetPres:eq2}
\Det _{(n-1)\,k}(S(L^{-}(X))) = \Det _{(n-1)\,k}(L^{-}(X))
\end{equation}
because $S$ preserves $\Det _{(n-1)\,k}$.

Lemma~\ref{lem:DetNKLPlus} implies that
%
\begin{equation}
\label{lem:LiftUpDetPres:eq3}
\Det _{(n-1)\,k}(S(L^{-}(X))) = \Det _{n\,k}(L^{+}(S(L^{-}(X)))).
\end{equation}
The equalities~\eqref{lem:LiftUpDetPres:eq1},
\eqref{lem:LiftUpDetPres:eq2} and~\eqref{lem:LiftUpDetPres:eq3} aligned
together imply that
\begin{equation*}
\Det _{n\,k}(X) = \Det _{n\,k}(L^{+}(S(L^{-}(X))))\;\;\text{for all}\;
\; X \in {\mathcal M}_{n\,k}({\mathbb F})
\end{equation*}
Since $ L^{+} \circ S \circ L^{-} = T$ by the definition of $T$, we conclude
that $T$ preserves $\Det _{n\,k}$.
\end{proof}

\begin{lemma}%
\label{lem:FromNKOddToNKEvenIfKerW}%
Assume that $|{\mathbb F}| > k$, $n > k \ge 2$ and $n + k$ is odd. Let
$T\colon {\mathcal M}_{n\, k} ({\mathbb F}) \to {\mathcal M}_{n\, k} ({
\mathbb F})$ be a linear map such that
$\Det _{n\, k} (T(X)) = \Det _{n\,k}(X)$ and
$\Img (T) \subseteq {\mathcal M}^{0}_{n\,k}({\mathbb F})$. Then the following
statements hold.
%
%
\begin{enumerate}[label=(\alph*),ref=(\alph*)]
\item
\label{lem:FromNKOddToNKEvenIfKerW:part4}%
$\Ker (T) = \rad (\Det _{n\,k})$.
\item
\label{lem:FromNKOddToNKEvenIfKerW:part5}%
$L^{+}\circ L^{-} \circ T = T$.
\item[$(c)$]
\label{lem:FromNKOddToNKEvenIfKerW:part6}%
$T \circ L^{+} \circ L^{-} = T$.
\item[$(d)$]
\label{lem:FromNKOddToNKEvenIfKerW:part2}%
If a linear map $S$ is defined by
%
\begin{equation}
\label{lem:FromNKOddToNKEvenIfKerW:eq}
S = L^{-} \circ T \circ L^{+},
\end{equation}
then $T = L^{+} \circ S \circ L^{-}$.
\end{enumerate}
\end{lemma}
\begin{proof}%
$(a)$.
The obvious equality
${\mathcal M}^{0}_{n\,k}({\mathbb F}) \cap \rad (\Det _{n\,k}) \allowbreak = \{0
\}$ implies that %
$T(\rad (\Det _{n\,k}))\allowbreak = \{0\}$ because
\begin{equation*}
T(\rad (\Det _{n\,k})) \subseteq \Img (T) \subseteq {\mathcal M}^{0}_{n
\,k}({\mathbb F})
\end{equation*}
and $T(\rad (\Det _{n\,k})) \subseteq \rad (\Det _{n\,k})$ by Lemma~\ref{lem:RadicalInvariant}(a).

Therefore, $\Ker (T) \supseteq \rad (\Det _{n\,k})$. Since
$\Ker (T) \subseteq \rad (\Det _{n\,k})$ by Lemma~\ref{lem:KernelContainedInRad},
then we conclude that $\Ker (T) = \rad (\Det _{n\,k})$.

\medskip
\noindent
$(b)$.
Indeed, if $X \in {\mathcal M}_{n\,k}({\mathbb F})$, then
\begin{equation*}
T(X) =
\begin{pmatrix}
(T(X))[1|)
\\
\vdots
\\
(T(X))[n-1|)
\\
0%
\end{pmatrix}
\end{equation*}
because $\Img (T) \subseteq {\mathcal M}^{0}_{n\,k}({\mathbb F})$. Hence,
\begin{multline*}
L^{+}\circ L^{-} \circ T (X) = L^{+}(L^{-}(T(X))) = L^{+}(L^{-}(
\begin{pmatrix}
(T(X))[1|)
\\
\vdots
\\
(T(X))[n-1|)
\\
0
\end{pmatrix}))
\\
= L^{+}(
\begin{pmatrix}
(T(X))[1|)
\\
\vdots
\\
(T(X))[n-1|)
\end{pmatrix}) =
\begin{pmatrix}
(T(X))[1|)
\\
\vdots
\\
(T(X))[n-1|)
\\
0
\end{pmatrix} = T(X) \;\;\mbox{for all}\;\;X \in {\mathcal M}_{n\,k}({\mathbb F}).
\end{multline*}
\medskip
\noindent
$(c)$.
On the one hand,
\begin{multline*}
T \circ L^{+} \circ L^{-} (X) = T(L^{+}(L^{-}(X)))
\\
= T(L^{+}(
\begin{pmatrix}
X[1|) - X[n|)
\\
\vdots
\\
X[n-1|) - X[n|)
\end{pmatrix})) = T(
\begin{pmatrix}
X[1|) - X[n|)
\\
\vdots
\\
X[n-1|) - X[n|)
\\
0
\end{pmatrix})\;\;\mbox{for all}\;\;X \in {\mathcal M}_{n\,k}({\mathbb F}).
\end{multline*}
On the other hand,
\begin{equation*}
T(X) = T(
\begin{pmatrix}
X[1|) - X[n|)
\\
\vdots
\\
X[n-1|) - X[n|)
\\
0%
\end{pmatrix}
) + T(
\begin{pmatrix}
X[n|)
\\
\vdots
\\
X[n|)
\\
X[n|)%
\end{pmatrix}
)
= T(
\begin{pmatrix}
X[1|) - X[n|)
\\
\vdots
\\
X[n-1|) - X[n|)
\\
0%
\end{pmatrix}
)
\end{equation*}
since
$
\begin{pmatrix}
X[n|)
\\
\vdots
\\
X[n|)
\\
X[n|)%
\end{pmatrix}
\in W_{n\,k}$ and consequently
$T(
\begin{pmatrix}
X[n|)
\\
\vdots
\\
X[n|)
\\
X[n|)%
\end{pmatrix}
) = O_{n\,k}$ by~(a).

\medskip
\noindent
$(d)$.
This part is a direct consequence from~(b)
and~(c). Indeed,
\begin{multline*}
L^{+} \circ S \circ L^{-}
\overset{\text{\eqref{lem:FromNKOddToNKEvenIfKerW:eq}}}{\GGG} L^{+}
\circ \left (L^{-} \circ T \circ L^{+} \right )\circ L^{-}
\\
= \left (L^{+} \circ L^{-} \circ T\right ) \circ L^{+} \circ L^{-}
\overset{\text{{(b)}}}{=} T \circ L^{+}
\circ L^{-} \overset{\text{{(c)}}}{=} T.\qedhere\quad\quad
\end{multline*}
\end{proof}

\begin{lemma}%
\label{lem:MatRepLift}%
Assume that $|{\mathbb F}| > k$, $n > k \ge 1$ and $n + k$ is odd. Let
$S$ be a linear map on ${\mathcal M}_{(n-1)\,k}({\mathbb F})$ and
$T$ be a linear map on ${\mathcal M}_{n\,k}({\mathbb F})$ defined by
\begin{equation*}
T = L^{+}\circ S \circ L^{-}.
\end{equation*}
If there exist $A' \in {\mathcal M}_{(n-1)\,(n-1)}({\mathbb F})$ and
$B' \in {\mathcal M}_{k\,k}({\mathbb F})$ satisfying the condition~\eqref{eq:TwoSideMul2}
such that
\begin{equation*}
S(Y) = A'YB'\;\;\text{for all}\;\; Y \in {\mathcal M}_{(n-1)\,k}({
\mathbb F}),
\end{equation*}
then there exist $A \in {\mathcal M}_{n\, n}({\mathbb F})$ and
$B \in {\mathcal M}_{k\, k}({\mathbb F})$ satisfying the condition~\eqref{eq:TwoSideMul2}
such that
\begin{equation*}
T(X) = AXB\;\;\text{for all}\;\; X \in {\mathcal M}_{n\,k}({\mathbb F}).
\end{equation*}
\end{lemma}
\begin{proof}%
Since $A'$ and $B'$ satisfy the condition~\eqref{eq:TwoSideMul2}, then
$S$ is a linear map preserving $\Det _{(n-1)\,k}$ by Lemma~\ref{lem:TwoSidedMulPreservesDet}.
Hence $T$ is a linear map preserving $\Det _{n\,k}$ by Lemma~\ref{lem:LiftUpDetPres}.

The equality $T = L^{+} \circ S \circ L^{-}$ implies that
\begin{equation*}
T(X) = M^{+}(S(M^{-}X))
\end{equation*}
for all $X \in {\mathcal M}_{n\,k}({\mathbb F})$. Recall there that
$M^{+}$ and $M^{-}$ are defined in Lemma~\ref{claim:LminusMatrixForm} and
Lemma~\ref{claim:LplusMatrixForm}. Therefore,
\begin{equation*}
T(X) = M^{+}(A'(M^{-}X)B') = (M^{+}A'M^{-})XB' = AXB
\end{equation*}
for all $X \in {\mathcal M}_{n\,k}({\mathbb F})$, where
$A\in {\mathcal M}_{n\,n}({\mathbb F})$ and
$B\in {\mathcal M}_{k\,k}({\mathbb F})$ are defined by
\begin{equation*}
A = M^{+}A'M^{-}, \quad B = B'.
\end{equation*}
Since $T$ is a linear map preserving $\Det _{n\,k}$ as we have shown above,
then $A$ and $B$ satisfy the condition~\eqref{eq:TwoSideMul2} by Lemma~\ref{lem:TwoSidedMulPreservesDet}.
\end{proof}

\begin{theorem}%
\label{thm:MainTheoremCullisNKOdd}%
Assume that $|{\mathbb F}| > k \ge 4$, $n \ge k + 2$ and $n + k$ is odd.
Let
$T\colon {\mathcal M}_{n\,k} ({\mathbb F}) \to {\mathcal M}_{n\,k} ({
\mathbb F})$ be a linear map. Then
$\Det _{n\, k} (T(X)) = \Det _{n\,k}(X)$ for all
$X \in {\mathcal M}_{n\,k} ({\mathbb F})$ if and only if there exist
$A \in {\mathcal M}_{n\, n}({\mathbb F})$ and
$B \in {\mathcal M}_{k\, k}({\mathbb F})$ such that
%
\begin{equation}
\label{thm:MainTheoremCullisNKOdd:eqq}
\Det _{n\,k} \Bigl(A(|i_{1},\ldots , i_{k}]\Bigr) \Det _{k} \Bigl(B
\Bigr) = (-1)^{\sum_{\alpha=1}^{k} (i_{\alpha} -\alpha)}\;\;\mbox{for all}\;\; 1 \le i_{1} < \ldots < i_{k} \le n
\end{equation}
and a linear map
$\phi \colon {\mathcal M}_{n\,k}({\mathbb F}) \to W_{n\,k}$ such that
%
\begin{equation}
\label{thm:MainTheoremCullisNKOdd:eq}
T(X) = AXB + \phi (X)\;\;\text{for all}\;\; X \in {\mathcal M}_{n\,k} ({
\mathbb F}).
\end{equation}
Here $W_{n\,k} \subseteq {\mathcal M}_{n\,k}({\mathbb F})$ denotes the
space of matrices, all rows of which are equal, which is defined in Definition~\ref{def:WNK}
and is a radical of $\Det _{n\,k}$ by Lemma~\ref{lem:NKOddRadDetNK}.
\end{theorem}
\begin{proof}
Let us show the sufficiency. Let $T$ be a linear map which satisfies the
condition~\eqref{thm:MainTheoremCullisNKOdd:eq} for certain
$A \in {\mathcal M}_{n\, n}({\mathbb F})$,
$B \in {\mathcal M}_{k\, k}({\mathbb F})$ satisfying the condition~\eqref{thm:MainTheoremCullisNKOdd:eqq}
and a linear map
$\phi \colon {\mathcal M}_{n\,k}({\mathbb F}) \to W_{n\,k}$.

Lemma~\ref{lem:TwoSidedMulPreservesDet} implies that the linear map
$T - \phi $ preserves $\Det _{n\,k}$. Therefore,
$T = \left (T - \phi \right ) + \phi $ preserves $\Det _{n\,k}$ by Corollary~\ref{cor:AdditionLinearMapNPLusKOdd}.

Let us prove the necessity. Assume that
$T\colon {\mathcal M}_{n\,k} ({\mathbb F}) \to {\mathcal M}_{n\,k} ({
\mathbb F})$ is a linear map preserving $\Det _{n\, k}$. Let
$\phi \colon {\mathcal M}_{n\,k}({\mathbb F}) \to W_{n\,k}$ be a linear
map defined by
\begin{equation*}
\phi (X) =
\begin{pmatrix}
\bigl(T(X)\bigr)[n|)
\\
\vdots
\\
\bigl(T(X)\bigr)[n|)%
\end{pmatrix}
,
\end{equation*}
$\tilde{T}\colon {\mathcal M}_{n\,k} ({\mathbb F}) \to {\mathcal M}_{n
\,k} ({\mathbb F})$ be a linear map defined by
$\tilde{T} = T - \phi $ and $S$ be a linear map on
${\mathcal M}_{(n-1)\,k}({\mathbb F})$ defined by
$S = L^{-}\circ \tilde{T} \circ L^{+}$.

The definition of $\tilde{T}$ implies that
$\Img (\tilde{T}) \subseteq {\mathcal M}^{0}_{n\,k}({\mathbb F})$. In addition,
$\tilde{T}$ preserves $\Det _{n\,k}$ by Corollary~\ref{cor:AdditionLinearMapNPLusKOdd}
applied to $T$ and $(-\phi )$. Thus, $\tilde{T}$ satisfies the conditions
of Lemma~\ref{lem:FromNKOddToNKEvenIfKerW}. In particular,
%
\begin{equation}
\label{thm:MainTheoremCullisNKOdd:eq1}
\tilde{T} = L^{+}\circ S \circ L^{-}
\end{equation}
by Lemma~\ref{lem:FromNKOddToNKEvenIfKerW}(d).

Since $\tilde{T}$ preserves $\Det _{n\,k}$, then $S$ preserves
$\Det _{(n-1)\,k}$ by Lemma~\ref{lem:LiftDownDetPres}. In addition,
$n > k + 2$ from the initial assumption that $n \ge k + 2$ and
$n + k$ is odd. This implies that $n - 1 \ge k + 2$. Hence, by Lemma~\ref{lem:MainTheoremEven}
we obtain that there exist
$A' \in {\mathcal M}_{(n-1)\,(n-1)}({\mathbb F})$ and
$B' \in {\mathcal M}_{k\,k}({\mathbb F})$ satisfying the condition~\eqref{eq:TwoSideMul2}
such that
%
\begin{equation}
\label{thm:MainTheoremCullisNKOdd:eq2}
S(Y) = A'YB'\;\;\text{for all}\;\; Y \in {\mathcal M}_{(n-1)\,k}({
\mathbb F}).
\end{equation}

Thus, the equalities~\eqref{thm:MainTheoremCullisNKOdd:eq1} and~\eqref{thm:MainTheoremCullisNKOdd:eq2}
allow us to conclude from Lemma~\ref{lem:MatRepLift} that there exist
$A \in {\mathcal M}_{n\, n}({\mathbb F})$ and
$B \in {\mathcal M}_{k\, k}({\mathbb F})$ satisfying the condition~\eqref{eq:TwoSideMul2}
such that
\begin{equation*}
\tilde{T}(X) = AXB\;\;\text{for all}\;\; X \in {\mathcal M}_{n\,k}({
\mathbb F}).
\end{equation*}
Therefore,
\begin{equation*}
T(X) = \tilde{T}(X) + \phi (X) = AXB + \phi (X)\;\;\mbox{for all}\;\;X \in {\mathcal M}_{n\,k} ({\mathbb F}).
\end{equation*}
\end{proof}

\begin{remark}%
\label{rem4.13}
If $T$ is a linear map on ${\mathcal M}_{n\,k}$ preserving
$\Det _{n\,k}$, then the representation~\eqref{thm:MainTheoremCullisNKOdd:eq}
may be not unique. For example,
\begin{equation*}
I_{n}XI_{k} +
\begin{pmatrix}
x_{1\,1} & \cdots & x_{1\,k}
\\
\vdots & \ddots & \vdots
\\
x_{1\,1} & \cdots & x_{1\,k}%
\end{pmatrix}
= AXI_{k}\;\;\text{for all}\;\; X \in {\mathcal M}_{n\,k}({\mathbb F}),
\end{equation*}
where $A = I_{n} + E_{1\,1} + \ldots + E_{n\,1}$, and the pair of matrices
$(A, I_{k})$ satisfies the condition~\eqref{thm:MainTheoremCullisNKOdd:eqq}.
\end{remark}

\begin{remark}%
\label{rem4.14}
There exist linear maps of the form~\eqref{thm:MainTheoremCullisNKOdd:eq}
which do not admit a representation $X \mapsto A'XB'$ for any
$A' \in {\mathcal M}_{n\,n}$, $B' \in {\mathcal M}_{k\,k}$. For example
a linear map
$T'\colon {\mathcal M}_{n\,k}({\mathbb F}) \to {\mathcal M}_{n\,k}({
\mathbb F})$ defined by
\begin{equation*}
T'(X) = X +
\begin{pmatrix}
-x_{1\,1} & 0 & \cdots & 0
\\
\vdots & \vdots & \ddots & \vdots
\\
-x_{1\,1} & 0 & \cdots & 0
\end{pmatrix}
\;\;\text{for all}\;\; X \in {\mathcal M}_{n\,k}({\mathbb F})
\end{equation*}
has the form~\eqref{thm:MainTheoremCullisNKOdd:eq} but increases a matrix
rank because
\begin{equation*}
\rk \left (E_{1\,1} + E_{1\,2}\right ) = 1
\end{equation*}
and
\begin{equation*}
\rk \left (T'(E_{1\,1} + E_{1\,2})\right ) = \rk \left (-E_{2\,1} -
\ldots -E_{n\,1} + E_{1\,2}\right ) = 2.
\end{equation*}
This implies that there do not exist
$A' \in {\mathcal M}_{n\,n}({\mathbb F})$,
$B' \in {\mathcal M}_{k\,k}({\mathbb F})$ such that
%
\begin{equation}
\label{remeq1}
T'(X) = A'XB'\;\;\text{for all}\;\; X \in {\mathcal M}_{n\,k}({
\mathbb F}),
\end{equation}
since every map of the form~\eqref{remeq1} cannot increase the rank of
an argument.
\end{remark}

\begin{remark}%
\label{rem4.15}
Similarly to the case when $n$ is even (see Lemma~\ref{lem:CounterEXK2NEven}),
for odd $n$ there are linear maps on
${\mathcal M}_{n\,2}({\mathbb F})$ which preserve $\Det _{n\,2}$ but do
not have the form~\eqref{thm:MainTheoremCullisNKOdd:eq} as it is shown
in Lemma~\ref{lem:CounterEXK2NOdd} below.
\end{remark}

The proof of Lemma~\ref{lem:CounterEXK2NOdd} relies on Lemma~\ref{lem:CounterEXK2NEven}
which is proven in~\cite{Guterman2025}.

\begin{lemma}[{\cite[Lemma~7.2]{Guterman2025}}]%
\label{lem:CounterEXK2NEven}%
Assume that $n$ is an even integer and $n \ge 4$. Then a linear map
$\mathbf{T}' \colon {\mathcal M}_{n\,2}({\mathbb F}) \to {\mathcal M}_{n
\,2}({\mathbb F})$ defined by
\begin{equation*}
\mathbf{T}'(X) =
\begin{pmatrix}
\sum \limits ^{n-1}_{i=2} (-1)^{i} x_{i\,1} + \sum \limits ^{n}_{i=2} (-1)^{i}
x_{i\,2} & x_{1\,2}
\\
x_{2\,1} & x_{2\,2}
\\
\vdots & \vdots
\\
x_{(n-1)\,1} & x_{(n-1)\,2}
\\
x_{n\,1} & \sum \limits ^{n-1}_{i=2} (-1)^{i-1} x_{i\,2} + \sum
\limits ^{n-1}_{i=1} (-1)^{i-1} x_{i\,1}
\end{pmatrix}
\end{equation*}
is a linear map preserving $\Det _{n\,2}$ which does not have the form
$X \mapsto AXB$ for any $A \in {\mathcal M}_{n\,n}({\mathbb F})$,
$B \in {\mathcal M}_{2\,2}({\mathbb F})$.
\end{lemma}

\begin{lemma}%
\label{lem:CounterEXK2NOdd}%
Assume that $n \ge 5$ is odd. Let
$\mathbf{T'}\colon {\mathcal M}_{(n-1)\,2}({\mathbb F}) \to {
\mathcal M}_{(n-1)\,2}({\mathbb F})$ be a linear map defined in the statement
of~\cite[Lemma~7.2]{Guterman2025}. Then
$\mathbf{T''} = L^{+} \circ \mathbf{T'} \circ L^{-}$ is a linear map preserving
$\Det _{n\,2}$ which does not have the form~\eqref{thm:MainTheoremCullisNKOdd:eq}.
\end{lemma}
\begin{proof}%
Indeed, $\mathbf{T'}$ preserves $\Det _{(n-1)\,2}$ by Lemma~\ref{lem:CounterEXK2NEven}.
Hence, $\mathbf{T''}$ preserves $\Det _{n\,2}$ by Lemma~\ref{lem:LiftUpDetPres}.

Suppose that there exist
$A'' \in {\mathcal M}_{n\,n}({\mathbb F}), B'' \in {\mathcal M}_{2\,2}({
\mathbb F})$ and
$\phi \colon {\mathcal M}_{n\,2}({\mathbb F}) \to W_{n\,2}$ such that
%
\begin{equation}
\label{lem:CounterEXK2NOdd:eq1}
\mathbf{T''}(X) = A''XB'' + \phi (X)\;\;\text{for all}\;\;X \in {
\mathcal M}_{n\,2}({\mathbb F}).
\end{equation}

On the one hand, since $L^{-} \circ L^{+}$ is an identical map on
${\mathcal M}_{(n-1)\,2}({\mathbb F})$, then
%
\begin{equation}
\label{lem:CounterEXK2NOdd:eq2}
L^{-}\circ \mathbf{T''} \circ L^{+} =L^{-} \circ L^{+} \circ
\mathbf{T'} \circ L^{-} \circ L^{+} = \mathbf{T'}.
\end{equation}

On the other hand, since $\ker (L^{-}) = W_{n\,2}$, then
%
\begin{multline}
\label{lem:CounterEXK2NOdd:eq3}
(L^{-}\circ \mathbf{T''} \circ L^{+})(Y)
\overset{\text{\eqref{lem:CounterEXK2NOdd:eq1}}}{\GGGL } L^{-}(A''(M^{+}(Y))B''
+ \phi (M^{+}Y))
\\
= (M^{-}A''M^{+})YB''\;\;\text{for all}\;\;Y \in {\mathcal M}_{(n-1)\,2}({
\mathbb F}).
\end{multline}

By comparing the right-hand sides of~\eqref{lem:CounterEXK2NOdd:eq2} and~\eqref{lem:CounterEXK2NOdd:eq3}
we obtain that
%
\begin{equation}
\label{lem:CounterEXK2NOdd:eq4}
\mathbf{T'}(Y)= (M^{-}A''M^{+})YB''\;\;\text{for all}\;\;Y \in {
\mathcal M}_{(n-1)\,2}({\mathbb F}).
\end{equation}
The equality~\eqref{lem:CounterEXK2NOdd:eq4} means that
$\mathbf{T'}$ has the form $X \mapsto AXB$ for $A = M^{-}A''M^{+}$ and
$B = B''$. This contradicts Lemma~\ref{lem:CounterEXK2NEven}. Therefore,
$\mathbf{T''}$ does not have the desired form.
\end{proof}

\section*{Declaration of competing interest}

The authors declare that they have no known competing financial interests or personal relationships that could have appeared to influence the work reported in this paper.

\section*{Acknowledgements}

We thank the anonymous referee for the careful reading, many helpful comments and suggestions. The research of the second author was supported by the scholarship of the Center for Absorption in Science, the Ministry for Absorption of Aliyah, the State of Israel.

\section*{Data availability}

No data was used for the research described in the article.

\bibliographystyle{plain}

\end{document}